\documentclass[reqno,12pt,letterpaper]{amsart}
\usepackage{amsmath,amssymb,amsthm,graphicx,mathrsfs,url}
\usepackage[usenames,dvipsnames]{color}
\usepackage{hyperref}
\hypersetup{colorlinks=true,linkcolor=Red,citecolor=Green}
\usepackage{amsxtra}
\usepackage[all]{xy}

\def\?[#1]{\textbf{[#1]}\marginpar{\Large{\textbf{??}}}}

\let\epsilon=\varepsilon 

\newcommand{\cM}{\mathcal M}
\newcommand{\R}{\mathbb R}

\newcommand{\C}{\mathbb C}

\newcommand{\cinf}{C^{\infty}}
\newcommand{\dd}{\partial}

\newcommand{\sgn}{\mathrm{sgn}\,}

\newcommand{\proj}{\mathrm{proj}}
\newcommand{\tr}{\mathrm{tr}}
\newcommand{\la}{\langle}
\newcommand{\ra}{\rangle}

\newcommand{\Di}{\mathcal{D}}
\newcommand{\Res}{\mathrm{Res}}
\newcommand{\im}{\mathrm{Im\,}}
\newcommand{\re}{\mathrm{Re\,}}
\newcommand{\dvol}{d\mathrm{vol}}

\theoremstyle{plain}
\newtheorem{thm}{Theorem}

\newtheorem{lem}[thm]{Lemma}
\newtheorem{prop}[thm]{Proposition}
\newtheorem{cor}[thm]{Corollary}

\theoremstyle{definition}
\newtheorem{defn}{Definition}[section]

\theoremstyle{remark}
\newtheorem*{rem}{Remark}

\title{Dynamical Zeta Functions in the Nonorientable Case}
\author{Yonah Borns-Weil}
\address{Department of Mathematics, University of California, Berkeley, CA 94720}
\email{yonah\_borns-weil@berkeley.edu}

\author{Shu Shen}
 \address{Institut de Math\'ematiques de Jussieu-Paris Rive Gauche, 
Sorbonne Universit\'e,  4 place Jussieu, 75252 Paris Cedex 5, France.}
\email{shu.shen@imj-prg.fr}

\begin{document}
\maketitle

\begin{abstract}
    We use a simple argument to extend the microlocal proofs of meromorphicity of dynamical zeta functions to the nonorientable case. In the special case of geodesic flow on a connected non-orientable negatively curved  closed surface, we compute the order of vanishing of the zeta function at the zero point to be the first Betti number of the surface.
\end{abstract}

\section{Background}

In this note we use a simple geometric argument to extend the results of Dyatlov and Zworski \cite{dz,dzsurfaces} and of Dyatlov and Guillarmou \cite{dglong,dgshort} to Axiom A flows with nonorientable stable and unstable bundles. It is classically known that on a closed manifold there are countably many closed orbits of such flows, and therefore one can define the \emph{Ruelle zeta function} $$\zeta_R(\lambda)=\prod_{\gamma^{\sharp}}\left(1-e^{i\lambda T_{\gamma}^{\sharp}}\right)$$ where the product is taken over all primitive closed geodesics $\gamma^{\sharp}$ with corresponding periods $T_{\gamma}^{\sharp}$.
Note that by \cite[Lemma 1.17]{dglong} and \cite[Section 3]{dgshort}, this product converges for $\im(\lambda)\gg1$ large enough. The meromorphic continuation of $\zeta_R$ to all of $\C$ was conjectured by Smale \cite{smale}, and proved by Fried \cite{fried} under analyticity assumptions. The case of smooth Anosov flows was first answered by Giulietti, Liverani and Policott \cite{glp} and then with microlocal methods by Dyatlov and Zworski \cite{dz} for manifolds with orientable stable and unstable bundles, and was extended to Axiom A flows by Dyatlov and Guillarmou \cite{dglong,dgshort} under the same orientability assumptions. In \cite[Appendix B]{glp}, the authors also outlined ideas for removing the orientability assumptions.

We remove the orientability assumption and give a full proof for Axiom A flows.  Specifically, we shall show

\begin{thm}\label{main}
If $(\phi_t)_{t\in\R}$ is an Axiom A flow on a closed manifold, the Ruelle zeta function $\zeta_R$ extends to a meromorphic function on $\C$.
\end{thm}
The definition of an Axiom A flow is given as Definition \ref{aaf}.

We then restrict to the case of contact Anosov flow on a $3$-manifold, and study the order of vanishing of $\zeta_R$ at $\lambda=0$. An important example is when $M=S^*\Sigma$, the cosphere bundle of a  connected negatively curved closed surface $\Sigma$, and $(\phi_t)_{t\in \mathbb R}$ is geodesic flow \cite{Anosov67}. This problem was treated in \cite{dzsurfaces} in the case where the stable bundle is orientable, and it was shown that the order of vanishing is $b_1(M)-2$, where $b_1(M)$ is the first Betti number of $M$.

We shall show that for nonorientable stable bundle, the analogous result is the following:

\begin{thm}\label{contact}
Let $(\phi_t)_{t\in\R}$ be the Reeb flow on a connected contact closed $3$-manifold.  If $(\phi_t)_{t\in\R}$ is Anosov with nonorientable stable bundle $E_s$, the Ruelle zeta function has vanishing order at $\lambda=0$ equal to $b_1(\mathscr{O}(E_s))$, the dimension of the first de Rham cohomology  with coefficients in the orientation line bundle of $E_s$.
\end{thm}

The orientation line bundle is reviewed in Definition \ref{ob}.

 In the special case of the geodesic flow on  $M=S^*\Sigma$ with $\Sigma$ nonorientable, 
 the vanishing order at $\lambda=0$ is given by $b_1(\Sigma)$, as is shown in Proposition \ref{Gysin}. This is in contrast to the orientable case, in which it is $b_1(\Sigma)-2$. 

More precisely, let $\chi'(\Sigma)$ be the derived Euler characteristic of $\Sigma$, i.e., 
$$
\chi'(\Sigma)=\sum_{i=0}^2(-1)^ii b_i(\Sigma)=\begin{cases}
-b_1(\Sigma)+2, &\text{if  $\Sigma$ is orientable},\\
-b_1(\Sigma),&\text{otherwise}. 
\end{cases}
$$

\begin{cor}\label{c3}
If $(\phi_t)_{t\in\R}$ is the geodesic flow on the cosphere bundle of a  connected negatively curved closed surface (orientable or not), the Ruelle zeta function has vanishing order at $\lambda=0$ equal to $-\chi'(\Sigma)$.
\end{cor}

\numberwithin{thm}{section}

\subsection{Axiom A Flows}

Let $M$ be a compact manifold without boundary of dimension $n$, and let $(\phi_t)_{t\in\R}$ be a flow on $M$ generated by the vector field $V\in \cinf(M;TM)$.
\begin{defn}
A $\phi_t$-invariant set $K\subseteq M$ is called \emph{hyperbolic} for the flow $(\phi_t)_{t\in \mathbb R}$ if $V$ does not vanish on $K$ and for each $x\in K$ the tangent space $T_xM$ can be written as the direct sum $$T_xM=E_0(x)\oplus E_s(x)\oplus E_u(x)$$ where $E_0(x)=\text{span}(V(x))$, $E_s,E_u$ are continuous $\phi_t$-invariant vector bundles on $K$, and for some Riemannian metric $|\cdot|$, there are $C,\theta>0$ such that for all $t>0$, 
\begin{align}\label{ineqAnosov}
\begin{split}
|d\phi_t(x)v|_{\phi_t(x)}&\le Ce^{-\theta t}|v|_x\qquad v\in E_s(x)\\ |d\phi_{-t}(x)w|_{\phi_{-t}(x)}&\le Ce^{-\theta t}|w|_x\quad\,\,\,\, w\in E_u(x).
\end{split}
\end{align}
In the important case where all of $M$ is hyperbolic, we call $(\phi_t)_{t\in \mathbb R}$ an \emph{Anosov flow}.
\end{defn}


There is an analogous notion of hyperbolicity at fixed points.
\begin{defn}
A fixed point $x\in M$, i.e., $V(x)=0$, is 
called \emph{hyperbolic} if the differential $DV(x)$ has no eigenvalues with vanishing real part.
\end{defn}

A generalization of Anosov flows is the following, given first by Smale \cite[II.5 Definition 5.1]{smale}:

\begin{defn}\label{aaf}
The flow $(\phi_t)_{t\in\R}$ is called \emph{Axiom A} if 
\begin{enumerate}
    \item all fixed points of $(\phi_t)_{t\in \mathbb R}$ are hyperbolic,
    \item the closure $\mathcal{K}$ of the union of all closed orbits of $(\phi_t)_{t\in \mathbb R}$ is hyperbolic,
    \item the nonwandering set (\cite[Definition 2.2]{dgshort}) of $(\phi_t)_{t\in \mathbb R}$ is the disjoint union of the set of fixed points and $\mathcal{K}$. 
\end{enumerate}
\end{defn}

 We now recall the definition of a locally maximal set, given in \cite[Definition 2.4]{dgshort}.

\begin{defn}
A compact $\phi_t$-invariant set $K\subseteq M$ is called \emph{locally maximal} for $(\phi_t)_{t\in \mathbb R}$ if there is a neighborhood $V$ of $K$ such that $$K=\bigcap_{t\in\R}\phi_t(V).$$
\end{defn}

We may then state the key proposition, which generalises \cite[Proposition 3.1]{dgshort} to the case where $E_s$ or $E_u$ is not necessarily orientable on $\mathcal K$.

\begin{prop}\label{mainlem}
Let $K\subseteq M$ be a locally maximal hyperbolic set for $(\phi_t)_{t\in \mathbb R}$, and let $\zeta_K$ be defined as the Ruelle zeta function where we only take the product over trajectories in $K$. Then $\zeta_K$ has a continuation to a meromorphic function on all of $\C$.
\end{prop}

Theorem \ref{main} follows from Proposition \ref{mainlem}, as we may remark that by \cite[II.5 Theorem 5.2]{smale} we can write $\mathcal{K}=K_1\sqcup\dots\sqcup K_N$ with $K_j$ basic hyperbolic\footnote{These are locally maximal hyperbolic by definition (see \cite[Definition 2.5]{dgshort}).}.  Then the product $$\zeta_R(\lambda)=\prod_{j=1}^N\zeta_{K_j}(\lambda),$$ which a priori holds for $\im(\lambda)\gg1$, gives that $\zeta_R$ also has a meromorphic continuation to all of $\C$.

The goal of Section \ref{Sprop4} is to prove Proposition \ref{mainlem}.

\subsection{The Orientation Bundle}
To fix notation we recall the definition of transition functions of a vector bundle. Given a continuous real vector bundle $E$ of rank $k$ over a manifold $M$ with projection map $\pi$, let $U_{\alpha}, U_{\beta}\subseteq M$ be two small open sets with nonempty intersection, and let $\psi_{\alpha}:\pi^{-1}U_{\alpha}\to U_{\alpha}\times\R^n$, $\psi_{\beta}:\pi^{-1}U_{\beta}\to U_{\beta}\times\R^n$ be local trivializations. Then the map $\psi_{\alpha}\circ\psi_{\beta}^{-1}:(U_{\alpha}\cap U_{\beta})\times\R^n\to(U_{\alpha}\cap U_{\beta})\times\R^n$ is of the form $$\psi_{\alpha}\circ\psi_{\beta}^{-1}(p,v)=(p,\tau_{\alpha\beta}(p)v)$$ where $\tau_{\alpha\beta}\in C^0(U_{\alpha}\cap U_{\beta}, {\rm GL}_k(\R))$ is called a \emph{transition function}. If the local trivializations can be chosen such that $\tau_{\alpha\beta}$ are smooth, then $E$ is a smooth vector bundle. Similarly, if $\tau_{\alpha\beta}$ can be chosen to be locally constant functions, then $E$ is a flat vector bundle.

Furthermore, suppose we are given an open cover $(U_{\alpha})_{\alpha\in A}$ of $M$ together with a set of continuous (resp.~smooth, resp.~locally constant) ${\rm GL}_k(\R)$-valued functions $(\tau_{\alpha\beta})_{\substack{\alpha,\beta\in A\\ U_{\alpha}\cap U_{\beta}\neq\emptyset}}$ with $\tau_{\alpha\alpha}=I$ on $U_\alpha$. Then there exists a continuous (resp. smooth, resp. flat) vector bundle $E$ 
with transition functions $\tau_{\alpha\beta}$, provided the following \emph{triple product property} holds: $$\tau_{\alpha\beta}(p)\tau_{\beta\gamma}(p)\tau_{\gamma\alpha}(p)=I $$ for any $p\in U_{\alpha}\cap U_{\beta}\cap U_{\gamma}.$ 


\begin{defn}\label{ob}
If $E$ is a continuous (but not necessarily smooth) real vector bundle over $M$ with transition functions $\tau_{\alpha\beta}$, the \emph{orientation bundle} of $E$ is a smooth flat line bundle $\mathscr{O}(E)$ with transition functions $$\sigma_{\alpha\beta}(p)={\rm sgn} \det( \tau_{\alpha\beta}(p))=\begin{cases}1&\det(\tau_{\alpha\beta}(p))>0\\-1&\det(\tau_{\alpha\beta}(p))<0.\end{cases}$$
\end{defn}


Recall that if $f:M\to M$ is a map, we say $f$ \emph{lifts} to a  bundle map $F:E\to E$ if $\pi\circ F=f\circ\pi$.

Since $\mathscr{O}(E)$ is a flat vector bundle, using the associated flat connection, we can lift the flow $(\phi_t)_{t\in\R}$ to  a flow $(\widetilde{\Phi}_t)_{t\in\R}$ on $\mathscr{O}(E)$. If the flow $(\phi_t)_{t\in\R}$ on $M$ lifts to a flow $(\Phi_t)_{t\in\R}$ on $E$, if $\psi, \eta$ are distinct trivializations of $E$ near $p$, $\phi_t(p)$ respectively, and $\widetilde{\psi}, \widetilde{\eta}$ are trivializations of $\mathscr{O}(E)$ near $p$, $\phi_t(p)$ respectively, we have for $p\in M$ and  $l\in\mathscr{O}(E)_p$:
\begin{align}
    \label{eqwildephi}
\widetilde{\Phi}_t(l)=\widetilde{\eta}^{-1}\left(\phi_t(p),\sgn\left(\det\left.\left(\eta\Phi_t\psi^{-1}\right)\right|_p\right)\proj_2\widetilde{\psi}(l)\right),
\end{align}
where $\proj_2$ is the obvious projection to the second component.

\subsection{Geodesic flows}
Let $Z$ be a  negatively  curved closed Riemannian  manifold. Let 
$M=S^{*}Z$ be the cosphere bundle on $Z$. It is classical that the geodesic 
flow on $M$ is Anosov \cite{Anosov67}. 

Let $\pi: M\to Z$ be the canonical projection. For $x\in M$, we have a morphism of linear spaces 
	\begin{align}\label{eqTMtoZ0}
	    		\pi_{*}: T_{x}M\to T_{\pi(x)}Z. 
	\end{align}

The following proposition is classical \cite[Section 22]{Anosov67} and \cite[Proposition 6]{Klingenberg}. We include a proof for the sake of completeness. 

\begin{prop}\label{propEsTZ}
The morphism $\pi_{*}$ induces an isomorphism  of continuous 
vector bundles on $M$, 
	\begin{align}\label{eqTMtoZ}
		E_{s}\oplus  E_0\simeq \pi^{*}(TZ). 
	\end{align} 
\end{prop} 
\begin{proof}
Since both sides of \eqref{eqTMtoZ} have the same dimension, it is enough 
to show that $\pi_{*}|_{E_s \oplus E_0}$ is injective. We will  show this  using Jacobi fields. It is convenient to work on the sphere bundle $M'=SZ$. We identify  $M'$ 
	with $M$ via the Riemannian metric on $Z$. 	
	
We follow \cite[Section II.H]{Eberlein01}. 
Let $\cM$ be the total space of $TZ$. Denote still by $\pi:\cM\to Z$  the obvious projection. 
Let $T^{V}\cM\subset T\cM$ be 
the vertical subbundle of $T\cM$. The Levi-Civita connection on $TZ$ 
induces a horizontal subbundle $T^{H}\cM\subset T\cM$ of $T\cM$, so 
that 
	\begin{align}\label{eqTMHV}
	T\cM=	T^{V}\cM\oplus  T^{H}\cM. 
	\end{align} 
Since $T^{V}\cM\simeq \pi^{*}(TZ)$ and  $T^{H}\cM\simeq 
\pi^{*}(TZ)$, by \eqref{eqTMHV}, we can identify the smooth vector bundles,  
	\begin{align*}
	T\cM= \pi^{*}(TZ\oplus TZ). 
	\end{align*} 
	
For $x={(z,v)}\in \cM$, let $\gamma_{x}$ be the unique  geodesic  on $Z$ such that 
$(\gamma_{x}(0),\dot{\gamma}_{x}(0))=(z,v). $
For  $w\in T_{x}\cM$, let $J_{x,w}\in 
C^{\infty}(\gamma_{x},TZ|_{\gamma_{x}})$ be the unique Jacobi field along 
$\gamma_{x}$ such that 
$	\left(J_{x,w}(0),\dot{J}_{x,w}(0)\right)=w,$
where $\dot{J}_{x,w}$ is the covariant derivation of $J_{x,w}$ in the 
direction $\dot \gamma_x$. Recall that a Jacobi field $J$ is called stable, if there is $C>0$ such that for all $t\ge 0$, 
	\begin{align*}
		\left|J(t)\right|\le C. 
	\end{align*} 
By \cite[Proposition VI.A]{Eberlein01}, given $x\in \mathcal M$, for any $Y_{1}\in T_{z}Z$, there exists one and only one stable Jacobi 
field $J$ along $\gamma_x$ such that $J(0)=Y_{1}$. 

For $x=(z,v)\in M'$, we have
	\begin{align*}
		T_{x}M'=\{(Y_{1},Y_{2})\in T_{z}Z\oplus T_{z}Z: \langle Y_{2},v\rangle =0\}. 
	\end{align*} 
The morphism $\pi_{*}$ in \eqref{eqTMtoZ0} is just 
	\begin{align*}
		w\in T_{x}M'\to J_{x,w}(0)\in T_{z}Z. 
	\end{align*} 
By \cite[Proposition VI.B]{Eberlein01},  $w\in E_{s}(x)\oplus E_0(x)$ if and only if the  Jacobi fields 
$J_{x,w}$ is stable. By the uniqueness of stable Jacobi fields, we see that $\pi_*|_{E_s\oplus E_0}$ is injective. 
\end{proof} 

Since $E_0$ is a trivial line bundle, our proposition implies immediately:

\begin{cor}\label{cor1}
	We have the isomorphism of smooth flat line bundles 
	\begin{align*}
		\mathscr{O}(E_{s})\simeq \pi^{*}(\mathscr{O}(TZ)). 
	\end{align*} 
\end{cor}

\section{Proof of Proposition \ref{mainlem}}\label{Sprop4}
We use the notation of \cite{dz}. If $0\le k\le n-1$,  let $\mathcal{E}_0^k\subset \Lambda^k(T^*M)$ denote the subbundle of $k$-forms $\omega$ such that $\iota_V\omega=0$, where $\iota$ denotes interior multiplication.


Let $\widetilde{\mathcal{E}_0^k}=\mathcal{E}_0^k\otimes \mathscr{O}(E_s).$ We consider the pullback  $\phi_{-t}^*$ on sections of $\widetilde{\mathcal{E}_0^k}$. Note that the flow $(\phi_t)_{t\in\mathbb R}$ lifts to a flow $(\Phi_t)_{t\in\mathbb R}$ on $\mathcal{E}_0^k$. Indeed, for $p\in M$, $\omega\in \mathcal E_{0,p}^k$, $\Phi_t \omega\in \mathcal E_{0,\phi_t(p)}^k$ is defined for $v_1,\cdots, v_k\in T_{\phi_t(p)}M$ by 
\begin{align}\label{eqPhiPoin}
\Phi_t\omega(v_1,\dots,v_k)=\omega \left((d\phi_t|_p)^{-1}v_1,\cdots,(d\phi_t|_{p})^{-1}v_k\right).
\end{align}Note that from the above formula, it is easy to check that $\iota_V\Phi_t\omega=0$.
Recall also that the flow $(\phi_t)_{t\in\mathbb R}$ lifts to a flow  $\widetilde{\Phi}_t$ on $\mathscr{O}(E_s)$ (see \eqref{eqwildephi}). For a section $s$ in $\widetilde{\mathcal E_0^k}$, we have 
\begin{align}\label{eqpullb}
    \phi_{-t}^*s\,(p)=\left(\Phi_t\otimes \widetilde{\Phi}_t\right) \big(s(\phi_{-t}(p))\big). 
\end{align}





Let $\chi\in C^\infty(M)$ be a smooth function whose support is contained in a small neighborhood of $K$ such that $\chi(x)=1$ for all $x\in K$. We now invoke the Guillemin trace formula (see \cite[pp.\ 501-502]{guillemin}, \cite[Appendix B]{dz}, \cite[(4.6)]{dglong})
which says that the flat trace $\tr^{\flat}\left.\chi\phi_{-t}^*\chi\right|_{C^{\infty}\left(M;\widetilde{\mathcal{E}_0^k}\right)}$ is a distribution on $(0,\infty)$ given by 
\begin{align}
    \label{eqGuillemin}
\tr^{\flat}\left.\chi\phi_{-t}^*\chi\right|_{C^{\infty}\left(M;\widetilde{\mathcal{E}_0^k}\right)}=\sum_{\gamma\subset K}\frac{T_{\gamma}^{\sharp}\ \tr^{\widetilde{\mathcal{E}}_{0,y}^k}\left(\Phi_{T_{\gamma}}\otimes\widetilde{\Phi}_{T_{\gamma}}\right)}{|\det(I-\mathcal{P}_{\gamma})|}\delta_{t-T_{\gamma}},
\end{align}
where the sum is taken over all the periodic trajectories $\gamma$ in $K$ with period $T_\gamma$ and primitive period $T_{\gamma}^{\sharp}$, $y$ is any point on $\gamma$, and $\mathcal{P}_{\gamma}=d\phi_{-T_{\gamma}}|_{(E_s\oplus E_u)_y}$ is the linearized Poincar\'e map at $y$. Note that as trace and determinant are invariant under conjugation, the right hand side does not depend on $y$.

By \eqref{eqPhiPoin},  the trace of  $\Phi_{T_{\gamma}}$ on $\mathcal{E}_{0,y}^k$ is just $\tr\left(\bigwedge^k\mathcal{P}_{\gamma}\right)$. By \eqref{eqwildephi}, we may take trivializations $\psi, \widetilde{\psi}$ of $E_s$, $\mathscr{O}(E_s)$ in a neighborhood of $y$ and have the induced lifting  on $\mathscr{O}(E_s)$ to be $\sgn\left(\det\left.\left(\psi d\phi_{T_{\gamma}}|_{E_{s,y}}\psi^{-1}\right)\right|\right).$ By definition we get this to be equal to $$\sgn\left(\det\left. d\phi_{T_{\gamma}}\right|_{E_{s,y}}\right)=\sgn\left(\det\left. d\phi_{-T_{\gamma}}\right|_{E_{s,y}}\right)=\sgn\det\left(\left. \mathcal{P}_\gamma\right|_{E_s}\right),$$ and as it is a map between one dimensional spaces, the trace is given by that expression as well. 
By the above consideration, we can rewrite \eqref{eqGuillemin} as 
\begin{align}\label{eqGuillemin1}
    \tr^{\flat}\left.\chi\phi_{-t}^*\chi\right|_{C^{\infty}\left(M;\widetilde{\mathcal{E}_0^k}\right)}=\sum_{\gamma\subset K}\frac{T_{\gamma}^{\sharp}\ \tr(\bigwedge^k\mathcal P_\gamma)\sgn(\det \mathcal{ P}_\gamma|_{E_s})}{|\det(I-\mathcal{P}_{\gamma})|}\delta_{t-T_{\gamma}}.
\end{align}

Let us follow \cite[Section 3]{dgshort}. By \cite[Lemma 3.2]{dgshort}, we may and we will assume that near $K$, $(\phi_t)_{t\in \mathbb R}$
is an open hyperbolic system in the sense of \cite[Assumptions (A1)–(A4)]{dglong}. By \cite[Lemma 1.17]{dglong}, there is $C>0$ such that for all $t\ge 0$, 
\begin{align}
    \label{eqexp}
    |\{\gamma \text{ closed trajectory in } K :T_\gamma\le t\}|\le Ce^{Ct}. 
\end{align}
For ${\rm Im}(\lambda)\gg1$ big enough, set 
\begin{align}\label{eqzetak}
    \zeta_{K,k}(\lambda)=\exp\left(-\sum_{\gamma\subset K}\frac{T_{\gamma}^{\sharp}}{T_{\gamma}} \frac{\tr(\bigwedge^k\mathcal P_\gamma)\sgn(\det \mathcal{ P}_\gamma|_{E_s})}{|\det(I-\mathcal{P}_{\gamma})|} e^{i\lambda T_{\gamma}} \right). 
\end{align}

\begin{lem}\label{Lem21}
For ${\rm Im}(\lambda)\gg1$ big enough, we have 
\begin{align}\label{eqGui}
    \dd_{\lambda}\log\zeta_{K,k}(\lambda)=-i\int_0^{\infty}e^{i\lambda t}\tr^{\flat}\left.\chi\phi_{-t}^*\chi\right|_{C^{\infty}\left(M;\widetilde{\mathcal{E}_0^k}\right)}dt.
\end{align}
The function $\zeta_{K,k}(\lambda)$ has a holomorphic extension to $\mathbb C$. 
\end{lem}
\begin{proof}Let us first remark that by \eqref{eqGuillemin1} and \eqref{eqexp}, the right hand side of \eqref{eqGui} is well defined.  Taking a logarithm and differentiating \eqref{eqzetak} and using Guillemin trace formula \eqref{eqGuillemin1}, we get \eqref{eqGui}.  The last part of 
the lemma follows from the arguments of \cite[Section 4]{dglong}. 
\end{proof}

Recall that for ${\rm Im}(\lambda)\gg1$ big enough, we have 
\begin{align}\label{eqzetaKK}
    \zeta_K(\lambda)=\prod_{\gamma^\sharp\subset K}\left(1-e^{i\lambda T_\lambda^\sharp}\right)=\exp\left(-\sum_{\gamma\subset K}\frac{T_{\gamma}^{\sharp}}{T_{\gamma}} e^{i\lambda T_{\gamma}} \right). 
\end{align}
Proposition \ref{mainlem} is a consequence of the following lemma. This lemma was stated in \cite{bt}, but we restate and prove it for convenience.
\begin{lem}\label{lemdet}
The following identity of meromorphic functions on $\mathbb C$ holds,
\begin{align}\label{eqZetakZetak}
    \zeta_K(\lambda)=\prod_{k=0}^{n-1}\big(\zeta_{K,k}(\lambda)\big)^{(-1)^{k+\dim E_s}}. 
\end{align}
\end{lem}
\begin{proof}
Following  \cite[(2.4)-(2.5)]{dz}, since 
$\det(I-\mathcal P_\gamma)=\sum_{k=0}^{n-1}(-1)^k \tr\left(\bigwedge^k \mathcal P_\gamma\right)$,  by \eqref{eqzetak} and \eqref{eqzetaKK}, it is enough to show
\begin{align}\label{eqbbsign}
    |\det(I-\mathcal{P}_{\gamma})|=(-1)^{\dim E_s}\ \sgn\left(\det \left. \mathcal{P}_\gamma\right|_{E_s}\right)\det(I-\mathcal{P}_{\gamma}).
\end{align}
Remark that
\begin{align}\label{eqDet}
\begin{split}
\det(I-\mathcal{P}_{\gamma})&=\det(I-\mathcal{P}_{\gamma}|_{E_u})\det(I-\mathcal{P}_{\gamma}|_{E_s})\\
&=(-1)^{\dim E_s}\det(I-\mathcal{P}_{\gamma}|_{E_u})\det(I-\mathcal{P}_{\gamma}^{-1}|_{E_s})\det(\mathcal{P}_{\gamma}|_{E_s}).
\end{split}
\end{align}
As time is running in the negative direction, we have by (\ref{ineqAnosov}) that the eigenvalues $\lambda$ of $\mathcal{P}_{\gamma}|_{E_u}$ have $|\lambda|<1$, and the eigenvalues $\mu$ of $\mathcal{P}_{\gamma}^{-1}|_{E_s}$ have $|\mu|<1$. This gives any eigenvalues of $I-\mathcal{P}_{\gamma}|_{E_u}$ to be either $1-\lambda$ for $\lambda\in (-1, 1)$ or conjugate pairs $1-\lambda, 1-\overline{\lambda}$ when $\lambda$ is not real. In any case, we get by multiplying that $$\det(I-\mathcal{P}_{\gamma}|_{E_u})>0$$ and similarly $$\det(I-\mathcal{P}_{\gamma}^{-1}|_{E_s})>0.$$ Then taking signs in (\ref{eqDet}), we get \eqref{eqbbsign}.
\end{proof}

\begin{rem}
The key point of our argument is based on the smoothness of $\mathscr O(E_s)$. Thanks to this property, most of the analytic arguments in the proof of Proposition \ref{mainlem} are reduced to 
\cite{dglong}. 
In \cite[Section 2]{bt}, Baladi and Tsujii used the orientation bundle in a different way for the flow with discrete time.
\end{rem}


\numberwithin{thm}{section}

\section{Vanishing Order at Zero on a Contact $3$-Manifold}

In this section, we assume that $M$ is a connected closed   $3$-manifold with a contact form $\alpha$, and that $V$ is the associated Reeb vector field.
We suppose also that the flow $(\phi_t)_{t\in \mathbb R}$ of $V$ is Anosov. 
One such example would be when $M=S^*\Sigma$, the cosphere bundle of a connected  closed  surface  $\Sigma$ with negative (variable) curvature, and $(\phi_t)_{t\in \mathbb R}$ is geodesic flow.

The following result was proven in \cite{dzsurfaces}:

\begin{thm}
If $(\phi_t)_{t\in \mathbb R}$ is a contact Anosov flow on a  connected closed $3$-manifold with orientable $E_u$ and $E_s$, the Ruelle zeta function has vanishing order at $\lambda=0$ equal to $b_1(M)-2$, where $b_1(M)$ denotes the first Betti number of $M$.
\end{thm}

The goal of this section is to determine the order of vanishing of $\zeta_R$ at $0$ in the case that $E_s$, $E_u$ are not orientable, and hence give a proof of Theorem \ref{contact}. We remark that since a contact manifold is orientable, orientability of $E_s$ is equivalent to orientability of $E_u$. 

\subsection{The twisted cohomology} \label{sTwicoh}
Let us recall some background and general facts on the twisted cohomology of a flat vector bundle. Let $X$ be a closed manifold. Let $F$ be a flat vector bundle on $X$ with flat connection $\nabla$. It induces a sheaf $\mathcal F$ on $X$ defined by locally constant sections, i.e., if $U\subset X$ is an open set, then 
$$\mathcal F(U)=\{s\in C^\infty(U;F|_U ): \nabla s=0\}.$$
The twisted cohomology $H^\bullet(X; F)$ is defined by the cohomology of the sheaf $\mathcal F$ \cite[Section II.4.4]{Godement73}. They are the algebraic invariants which describe the rigidity properties of the global flat sections of $F$. Let $b_k(F)$ be the twisted Betti number 
$$b_k(F)=\dim H^k(X;F).$$
If $F$ is the trivial line bundle, we get the classical de Rham cohomology with real coefficients.

To evaluate $H^\bullet(X; F)$, one can use the twisted de Rham complex. Indeed, if we denote  $F^k=\Lambda^k(T^*X)\otimes F$,  the flat connection $\nabla$ extends to an operator  $d_k:C^\infty(X;F^k)\to C^\infty(X;F^{k+1})$ by Leibniz rule: if $\alpha\in C^\infty(X;\Lambda^k(T^*X))$ and $s\in C^\infty(X;F)$, we have 
$$d_k(\alpha\cdot s)=d\alpha \cdot s+(-1)^k \alpha \wedge \nabla s.$$
By the flatness of $\nabla$, we have $d_{k+1}d_k=0$, so that $(C^\infty(X;F^\bullet),d_\bullet)$ is a complex. By the de Rham isomorphism \cite[Théorème II.4.7.1]{Godement73}, we have
\begin{align}\label{eqHdd}
H^k(X;F)=\ker d_k/{\rm Im} d_{k-1}.    
\end{align}

As an analogue of \cite[Lemma 2.1]{dzsurfaces}, 
using the theory of elliptic operators, we can evaluate $H^\bullet(X;F)$ using the complex of twisted currents, or more generally twisted currents with wavefront conditions. 

More precisely, let $\Gamma\subset T^*X$ be a closed cone. We denote by $\Di'_{\Gamma}(X;F^k)$ the space of $F^k$-valued distributions whose wavefront set is contained in $\Gamma$ (see \cite[Section 2.1]{dzsurfaces}).  By microlocality, we have  $$
 d_k:\Di'_{\Gamma}\left(X;F^k\right)\to \Di'_{\Gamma}\left(X;F^{k+1}\right).$$
For simplicity, we will write $d$ sometimes.   




\begin{lem}\label{hodge}
If $u\in\Di'_{\Gamma}\left(X;F^k\right)$ and $du\in C^{\infty}\left(X;F^{k+1}\right)$, then there exist $v\in C^{\infty}\left(X;F^k\right)$ and $w\in\Di'_{\Gamma}\left(X;F^{k-1}\right)$ such that 
$$u=v+dw.$$
In particular, if $u\in\Di'_{\Gamma}\left(X;F\right)$ and $du\in C^{\infty}\left(X;F^1\right)$, then $u\in C^{\infty}\left(X;F\right).$
\end{lem}
\begin{proof}
Take a Riemannian metric on $X$ and a Hermitian metric on $F$. Remark that these two metrics induce a fibrewise scalar product  $\langle\cdot,\cdot\rangle$ on $F^k$. For $u,v\in C^\infty\left(X;F^k\right)$, 
we can define the $L^2$-product by 
\begin{align}\label{eqL2}
\la u,v\ra_{L^2\left(X;F^k\right)}=\int_X\la u,v\ra\,\dvol,
\end{align}
where $\dvol$ is a volume form. 
Let $\delta_{k+1}:C^\infty\left(X;F^{k+1}\right)\to C^\infty\left(X;F^k\right)$ be the formal adjoint of $d$ with respect to the  $L^2$-product \eqref{eqL2}. Define the twisted Hodge Laplacian by $$\Delta_k=d_{k-1}\delta_{k}+\delta_{k+1} d_k: C^\infty\left(X;F^k\right)\to C^\infty\left(X;F^k\right).$$ 
Then $\Delta_k$ is an essentially self-adjoint  second order elliptic differential operator. The remainder of the proof carries over identically from  that of \cite[Lemma 2.1]{dzsurfaces}.
\end{proof}

Remark that if $F$ is the orientation bundle of certain vector bundle and $u\in C^\infty(X;F)$, then for $x\in X$, $|u(x)|^2$ is independent of the choice of trivializations. It defines a Hermitian metric on $F$.

\subsection{ Resonant State Spaces.}
Let $M$ be a connected  $3$-dimensional closed manifold with a contact form $\alpha\in C^\infty(M,T^*M)$. Let $V$ be the associated Reeb vector field. Then, 
\begin{align}\label{eqReeb1}
i_V\alpha=1, \quad \iota_Vd\alpha=0. 
\end{align}
We assume that the flow $(\phi_t)_{t\in \mathbb R}$ associated to $V$ is  Anosov. Let $E^*_u\subset T^*M$ be the dual of $E_s$. We will apply the results of Section \ref{sTwicoh} to the case where $(X,F,\Gamma)=(M, \mathcal O(E_s),E^*_u).$

Since the flow $(\phi_t)_{t\in \mathbb R}$ is Anosov, we have $K=M$. For $0\le k\le 2$,  we write $\zeta_k=\zeta_{K,k}$. By \eqref{eqZetakZetak}, we have 
\begin{align}\label{EqZetaR}
\zeta_R(\lambda)=\frac{\zeta_1(\lambda)}{\zeta_0(\lambda)\zeta_2(\lambda)}.
\end{align}

We consider 
the operator 
$P_k=-i\mathcal{L}_V$, where $\mathcal{L}_V$ (in a slight abuse of notation) denotes the natural action on sections of $\widetilde{\mathcal{E}_0^k}$, given by the Lie derivative on sections of $\mathcal{E}_0^k$ tensored with the flat connection on $\mathscr{O}(E_s)$. For $\text{Im}\lambda\gg1$ large enough, the integral $R_k(\lambda)=i\int_0^\infty e^{i\lambda t}\phi_{-t}^* dt$
converges and defines a bounded operator on the $L^2$-space; this is nothing more than the resolvent operator of $P_k$.  Then by \cite[Section 2.3]{dzsurfaces} we have that $R_k$ extends meromorphically to the entire complex plane, 
$$R_k(\lambda): C^\infty\left(M;\widetilde{\mathcal E_0^k}\right)\to \mathcal{D}'\left(M;\widetilde{\mathcal E_0^k}\right).$$
More precisely, near $\lambda_0\in \mathbb C$,  we have 
 $$R_k(\lambda)=R_{k,H}(\lambda)-\sum_{j=1}^{J(\lambda_0)}\frac{(P_k-\lambda)^{j-1}\Pi_k}{(\lambda-\lambda_0)^j}$$ where $R_{k,H}$ is a holomorphic family defined near $\lambda_0$, $J(\lambda_0)\in \mathbb N$,  and $\Pi_k$ has rank $m_k(\lambda_0)<\infty$. By the arguments at the end of \cite{dz}, we have that at $\lambda_0$, the function  $\zeta_k$ has a zero of order $m_k(\lambda_0)$.

We define the space of \emph{resonant states} at $\lambda_0$ to be $$\Res_k(\lambda_0)=\left\{u\in\Di'_{E^*_u}\left(M;\widetilde{\mathcal{E}_0^k}\right):(P_k-\lambda_0)u=0\right\}.$$ 
Then a special case of \cite[Lemma 2.2]{dzsurfaces} gives the following:
\begin{lem}\label{semisimplicity}
Suppose $P_k$ satisfies the \emph{semisimplicity condition}: $$u\in\Di'_{E_u^*}\left(M;\widetilde{\mathcal{E}_0^k}\right),\quad (P_k-\lambda_0)^2u=0\,\,\implies \,\,(P_k-\lambda_0)u=0.$$ Then $m_k(\lambda_0)=\dim\Res _k(\lambda_0)$.
\end{lem}

Recall that we are trying to find the order at $\lambda=0$ of $\zeta_R$, which by (\ref{EqZetaR}) is simply 
\begin{align}\label{eqmR0}
m_R(0)=m_1(0)-m_0(0)-m_2(0).    
\end{align}
 We will compute each of these individually, by computing $\dim\Res_k(0)$ and checking that the semisimplicity condition in Lemma \ref{semisimplicity} holds.

We begin with twisted ``$0$-forms', which are just sections of the orientation bundle $\mathscr{O}(E_s)$.

\begin{prop}\label{0forms}
If $E_s$ is nonorientable, the space $\Res_0(0)$ is $\{0\}$.
\end{prop}
\begin{proof}
Suppose $u\in\Res_0(0)$, i.e., 
\begin{align}\label{eqP0}
    P_0u=0. 
\end{align}
Since the flow $(\phi_t)_{t\in\mathbb R}$ preserves the contact volume form $\alpha \wedge d\alpha$, $P_0: \cinf(M;\mathscr{O}(E_s))\to \cinf(M;\mathscr{O}(E_s))$ is a symmetric operator  with respect to the $L^2$-product \eqref{eqL2}.   By \cite[Lemma 2.3]{dzsurfaces}, $u\in C^{\infty}\left(M;\mathscr{O}(E_s)\right)$. Using $\dd_t\left(\phi_{-t}^*u\right)=-\phi_{-t}^*\nabla_Vu$ (where $\nabla$ is the flat connection), we see that
$u$ is constant on the flow line: for all $t\in \mathbb R$, 
\begin{align}\label{eqconfl}
u=\phi_{-t}^* u.
\end{align}
Let $(x,v)\in TM.$ The pairing $\la du(x), v\ra $ is an element of $\mathscr{O}(E_s)_x$. 
By \eqref{eqpullb} and \eqref{eqconfl}, we have
$$\la du(x), v\ra=\la \phi_{-t}^*(du)(x),v \ra= \widetilde{\Phi}_t \la du(\phi_{-t}(x)),d\phi_{-t}(x)v\ra.$$
If $v\in E_u(x)$, then sending $t\to\infty$ gives $\la du(x), v\ra=0$ by (\ref{ineqAnosov}). Similarly, if $v\in E_s(x)$, then sending $t\to-\infty$ gives $\la du(x), v\ra=0$. This shows that 
\begin{align}\label{eqduEUS}
du|_{E_s\oplus E_u}=0.
\end{align}

By Cartan's formula and by \eqref{eqP0}, we have $\iota_V du=0$, i.e., 
\begin{align}\label{eqduE0}
     du|_{E_0}=0. 
\end{align}
By \eqref{eqduEUS} and \eqref{eqduE0}, we have $du=0$. So $u\in H^0(M;\mathscr{O}(E_s))$. Since $E_s$ is nonorientable, we have $H^0(M;\mathscr{O}(E_s))=0$, so $u=0$ and $\Res_0(0)$ is trivial.
\end{proof}

\begin{cor}\label{0formcor}
If $E_s$ is nonorientable,  the multiplicity for $0$-forms is $m_0(0)=0$.
\end{cor}
\begin{proof}
If $P_0^2(u)=0$, then $P_0u\in\Res_0(0)$. By Proposition \ref{0forms}, $P_0u=0$, so $u\in\Res_0(0)$. This shows semisimplicity, so by Lemma \ref{semisimplicity} we see that $m_0(0)=\dim\Res_0(0)=0$.
\end{proof}

\begin{prop}\label{2forms}
If $E_s$ is nonorientable, the space $\Res_2(0)$ is $\{0\}$.
\end{prop}
\begin{proof}
We claim that 
\begin{align}\label{eqa2}
\alpha\wedge: \mathcal E^2_0\to\mathcal E^3
\end{align}
is a bundle isomorphism. Indeed, using \eqref{eqReeb1}, it is easy to see that the inverse of \eqref{eqa2} is given by $\iota_V$. Tensoring with $\mathscr O(E_s)$, we get a bundle isomorphism
\begin{align}\label{eqa3}
\alpha\wedge: \widetilde{\mathcal E^2_0}\to\widetilde{\mathcal E^3}. 
\end{align}

Let $u\in\Res_2(0)$. Since $\mathcal E^3$ is generated by $\alpha\wedge d\alpha$, by \eqref{eqa3}, there is $v\in\Di'_{E_u^*}(M;\mathscr{O}(E_s))$ such that 
$\alpha \wedge u= v \alpha\wedge d\alpha$. 
Applying $\iota_V$ and using $\iota_Vu=0$, we have $u=vd\alpha$. Then 
$$0=P_2(u)=(P_0v)d\alpha.$$
But this gives $P_0v=0$, so by Proposition \ref{0forms} we have $v=0$. Therefore, $u=0$.
\end{proof}

The following is then clear for the same reason as Corollary \ref{0formcor}.
\begin{cor}\label{cor2}
If $E_s$ is nonorientable, the multiplicity for $2$-forms is $m_2(0)=0$.
\end{cor}

We now turn to the case of $P_1$ acting on the space of twisted $1$-form-valued distributions $\Di'_{E_u^*}\left(M;\widetilde{\mathcal{E}_0^1}\right)$.
We can now state the analogous proposition for $1$-forms:
\begin{prop}\label{p38}
If $E_s$ is nonorientable, the space $\Res_1(0)$ has dimension $b_1(\mathscr{O}(E_s))$.
\end{prop}
\begin{proof}
The proof is analogous to that of \cite[Lemma 3.4]{dzsurfaces}, but slightly easier due to the holomorphy of the resolvent $R_0$ near $0$.
Let $u\in\Res_1(0)$. Then $du\in\Res_2(0)$ by Proposition \ref{2forms}, so $du=0$. By Lemma \ref{hodge} there is a $\phi\in\Di'_{E_u^*}(M;\mathscr{O}(E_s))$ such that
\begin{align*}
u-d\phi\in C^{\infty}\left(M;\widetilde{\mathcal{E}^1}\right),\quad d(u-d\phi)=0.
\end{align*} We shall show that the map: $$\Theta: u\mapsto[u-d\phi]\in H^1(M;\mathscr{O}(E_s))$$ is well-defined, linear and bijective, which is enough to prove the lemma.

\noindent\textbf{Well-Definedness and linearity}:

Suppose there is another section $\psi\in \Di'_{E_u^*}\left(M;\mathscr{O}(E_s)\right)$ with $u-d\psi\in C^{\infty}\left(M;\widetilde{\mathcal{E}^1}\right)$. Then subtracting gives $d(\phi-\psi)\in C^{\infty}\left(M;\widetilde{\mathcal{E}^1}\right)$, so $\phi-\psi\in C^{\infty}\left(M;\mathscr{O}(E_s)\right)$ by Lemma \ref{hodge}. This shows that the map $\Theta$ is well-defined. It is also easy to see that $\Theta$ is linear.

\noindent\textbf{Injectivity}:

If $\Theta(u)=0$, then $u-d\phi$ is exact, so without loss of generality we can assume that $u=d\phi$. Combining with  $\iota_V u=0$, we get $\phi\in\Res_0(0)$, so $\phi=0$ by Proposition  \ref{0forms}. Therefore $u=0$, and this shows $\Theta$ to be injective.

\noindent\textbf{Surjectivity}:

Let $v\in C^{\infty}\left(M;\widetilde{\mathcal{E}^1}\right)$ with $dv=0$. Then as $m_0(0)=0$, the resolvent $R_0$ is holomorphic near $0$. Take $\phi=i R_0(0)\iota_V v \in\Di'_{E_u^*}\left(M;\mathscr{O}(E_s)\right)$. Then $P_0\phi=i \iota_Vv$. This rearranges to $\iota_V(v+d\phi)=0$, so $v+d\phi\in\Res_1(0)$. This gives that $\Theta$ is surjective, and completes the proof of our proposition.
\end{proof}

\begin{prop}\label{prop3ma}
If $E_s$ is nonorientable, the multiplicity for $1$-forms is $m_1(0)=b_1(\mathscr{O}(E_s))$.
\end{prop}
\begin{proof}
By Lemma \ref{semisimplicity}, we must only check that the semisimplicity condition is satisfied. Take $u\in \mathcal D'_{E_u^*}\left(M;\widetilde{\mathcal E_0^1}\right)$ such that $(P_1)^2u=0$. Then $v=\iota_V du\in\Res_1(0)$. It is enough to show that  $v=0$.

Recall that in the proof of Proposition \ref{p38},  we have seen that elements in $ {\rm Res}_1(0)$ are closed. In particular, 
\begin{align}\label{eqdv00}
    dv=0. 
\end{align}

Note that $\alpha\wedge du\in \mathcal D'_{E_u^*}(M;\widetilde{\mathcal E^3})$. We claim that 
\begin{align}\label{eqadu0}
    \alpha\wedge du=0. 
\end{align}
Indeed, there is some $a\in\Di'_{E_u^*}(M;\mathscr{O}(E_s))$ such that 
\begin{align*}
     \alpha\wedge du=a\,  \alpha\wedge d\alpha. 
\end{align*}
Since $\mathcal{L}_V(\alpha)=0$, by \eqref{eqdv00}, we have
\begin{align*}
     ({\mathcal L}_Va) \alpha\wedge d\alpha=\alpha \wedge  {\mathcal L}_V (du)= \alpha\wedge d\iota_Vdu=\alpha\wedge dv=0. 
\end{align*}
Then $\mathcal{L}_Va=0$, so $a=0$ by Proposition  \ref{0forms}. This gives \eqref{eqadu0}.

Since $\alpha(V)=1$, we have  $(\alpha\wedge)\circ  \iota_V+\iota_V\circ(\alpha\wedge)=\mathrm{id}$. By \eqref{eqadu0}, we have 
\begin{align}\label{eqduav}
    du=((\alpha\wedge)\circ  \iota_V+\iota_V\circ(\alpha\wedge))  du=\alpha\wedge v. 
\end{align}

 By Lemma \ref{hodge} and by \eqref{eqdv00}, there are $w\in C^{\infty}\left(M;\widetilde{\mathcal{E}^1}\right)$, $\phi\in\Di'_{E_u^*}\left(M;\mathscr{O}(E_s)\right)$ such that 
 \begin{align}\label{eq27ma}
     v=w+d\phi,\quad dw=0.
 \end{align}
Then
\begin{align}\label{eq27}
    \iota_Vw=\iota_V(v-d\phi)=-\mathcal{L}_V\phi. 
\end{align}
In particular, $\mathcal{L}_V\phi$ is smooth. 
We compute by Stokes' Theorem and by \eqref{eqduav}-\eqref{eq27}, 
\begin{align*}
0&=\re\int_Mdu\wedge \overline{w}=\re\int_M\alpha\wedge d\phi\wedge \overline{w}=\re\int_M\phi \overline{w}\wedge d\alpha\\
&=\re\int_M\iota_V(\phi \overline{w})\,\alpha \wedge d\alpha=-\re\int_M \phi (\overline{\mathcal L_V\phi})\alpha \wedge d\alpha=-\re\la \mathcal L_V\phi,\phi\ra_{L^2\left(M;\mathscr{O}(E_s)\right)}, 
\end{align*}
where the fourth equality comes from 
the fact the 
$$(\alpha \wedge)\circ \iota_V (\phi \overline{w}\wedge d\alpha)=((\alpha \wedge) \circ \iota_V+ \iota_V \circ (\alpha \wedge) )(\phi \overline{w}\wedge d\alpha)=\phi \overline{w}\wedge d\alpha.$$
In the above formula, we use the fact that a product of two twisted forms is  untwisted.  By \cite[Lemma 2.3]{dzsurfaces} we have $\phi\in C^{\infty}\left(M;\mathscr{O}(E_s)\right)$, so $v\in C^{\infty}\left(M;\widetilde{\mathcal{E}_0^1}\right)$. Then by the same argument as in Proposition  \ref{0forms} (see \cite[Lemma 3.5]{dzsurfaces}) we have $v=0$. 
\end{proof}

Now Theorem \ref{contact} is a consequence of \eqref{eqmR0}, Corollaries \ref{0formcor}, \ref{cor2}, and Proposition \ref{prop3ma}.

Let $\Sigma$ be a connected  negatively curved closed surface. Take $M=S^* \Sigma$. By Corollary \ref{cor1}, we have 
	\begin{align*}
		H^{1}(M;\mathscr{O}(E_{s}))= H^{1}(M;\pi^{*}\mathscr{O}(T\Sigma)). 
	\end{align*} 
	
\begin{prop}\label{Gysin}
If $\Sigma$ is a connected negatively curved closed surface (oriented or not), we 
have 
	\begin{align}\label{eqHGY0}
		\dim H^{1}(M;\pi^{*}\mathscr{O}(T\Sigma))=\dim H^{1}(\Sigma). 
	\end{align} 
\end{prop} 	
\begin{proof}	
By the Gysin long exact sequence, we have the exact sequence
	\begin{align*}
\xymatrix{
0\ar[r]& H^{1}(\Sigma;\mathscr{O}(T\Sigma)) \ar[r]^-{\pi^{*}}  
&H^{1}(M;\pi^{*}\mathscr{O}(T\Sigma))\ar[r]^-{\pi_{*}}& H^{0}(\Sigma) \ar[r]^-{ e\wedge} & 
H^{2}(\Sigma;\mathscr{O}(T\Sigma))\ar[r]&,}
	\end{align*} 
where $\pi^{*}$ is the pullback, $\pi_{*}$ is the integration along 
the fibre of $M\to \Sigma$, and $e\in H^{2}(\Sigma;\mathscr{O}(T\Sigma))$ is the Euler class of $T\Sigma$. 

We claim that the last map  
	\begin{align*}
e\wedge		:H^{0}(\Sigma)\to H^{2}(\Sigma;\mathscr{O}(T\Sigma))
	\end{align*} 
	in the Gysin exact sequence is an isomorphism. Indeed, since $\Sigma$ is connected, we have 
	$\dim H^{0}(\Sigma)=1$, and by Poincar\'e duality, $\dim 
	H^{2}(\Sigma;\mathscr{O}(T\Sigma))=1$. It is enough to show that $e\in H^{2}(\Sigma;\mathscr{O}(T\Sigma))$ 
	is non zero, or equivalently $\int_{\Sigma}e\neq0$.  This is a consequence of the fact that $\Sigma$ has negative curvature, as $e=K\mu$ where $\mu$ is the Riemannian density and $K<0$ is the Gauss curvature.  

Therefore, we get an isomorphism  
	\begin{align}\label{eqHGY}
		\pi^{*}: H^{1}(\Sigma;\mathscr{O}(T\Sigma)) \simeq H^{1}(M;\pi^{*}\mathscr{O}(T\Sigma)). 
	\end{align} 
By Poincar\'e duality, we have 
	\begin{align}\label{eqHGY1}
		H^{1}(\Sigma;\mathscr{O}(T\Sigma))\simeq \left(H^{1}(\Sigma)\right)^{*}. 
	\end{align} 
By \eqref{eqHGY} and \eqref{eqHGY1}, we get \eqref{eqHGY0}. 
\end{proof}
Now Corollary \ref{c3} is a consequence of Theorem \ref{contact} and  Proposition \ref{Gysin}. 

\section{Acknowledgements}

We would like to thank Maciej Zworski for suggesting the problem and giving guidance along the way, and Kiran Luecke for some helpful commments. We are indebted to the referees for reading the paper very carefully and providing us the reference \cite{Klingenberg}. Y.~B-W also acknowledges the partial support under the NSF grant DMS-1952939. S.S. acknowledges the partial support under the ANR grant ANR-20-CE40-0017.

\bibliographystyle{acm}

\bibliography{main}

\end{document}